\documentclass[a4paper, 11pt]{amsart} 
\usepackage{enumitem, amsmath, amssymb, amsthm, xspace, caption, subcaption, comment, amsaddr}
\usepackage[margin=1.28in]{geometry}
\usepackage{tikz}
\usepackage{hyperref}

\newtheorem{thm}{Theorem}[section]
\newtheorem{lemma}[thm]{Lemma}
\newtheorem{prop}[thm]{Proposition}

\newtheorem{qu}[thm]{Question}

\newtheorem{claim}[thm]{Claim}

\newtheorem{obs}[thm]{Observation}

\theoremstyle{definition}

\theoremstyle{plain}

\newcommand{\sm}{\setminus}

\newcommand{\eps}{\varepsilon}
\newcommand{\vphi}{\varphi}

\newcommand{\Expect}{\mathbb{E}}

\newcommand{\genericScale}{1}
\newcommand{\genericThickness}{0.65mm}

\newcommand{\genericCirclePt}{3pt}

\newcommand{\cupdot}{\mathbin{\mathaccent\cdot\cup}}

\title[Contagious sets in bootstrap percolation]{Contagious sets in a degree-proportional bootstrap percolation process} 
\author{Frederik Garbe and Richard Mycroft}
\address[Frederik Garbe,Richard Mycroft]{School of Mathematics,\\ University of Birmingham,\\ Birmingham B15 2TT,\\ United Kingdom}
\author{Andrew McDowell}
\address[Andrew McDowell]{Informatics Department,\\ King's College London,\\ London WC2R 2LS,\\ United Kingdom}
\email{fxg472@bham.ac.uk}
\email{andrew.mcdowell@kcl.ac.uk}
\email{r.mycroft@bham.ac.uk}
\thanks{RM supported by EPSRC grant EP/M011771/1.}

\begin{document}

\begin{abstract}
We study the following bootstrap percolation process: given a connected graph $G$, a constant $\rho \in [0, 1]$ and an initial set $A \subseteq V(G)$ of \emph{infected} vertices, at each step a vertex~$v$ becomes infected if at least a $\rho$-proportion of its neighbours are already infected (once infected, a vertex remains infected forever). Our focus is on the size $h_\rho(G)$ of a smallest initial set which is \emph{contagious}, meaning that this process results in the infection of every vertex of $G$.

Our main result states that every connected graph $G$ on $n$ vertices has $h_\rho(G) < 2\rho n$ or $h_\rho(G) = 1$ (note that allowing the latter possibility is necessary because of the case $\rho\leq\tfrac{1}{2n}$, as every contagious set has size at least one). This is the best-possible bound of this form, and improves on previous results of Chang and Lyuu and of Gentner and Rautenbach. We also provide a stronger bound for graphs of girth at least five and sufficiently small $\rho$, which is asymptotically best-possible.
\end{abstract}

\clearpage\maketitle
\thispagestyle{empty}

\section{Introduction}
We study a graph percolation process in which a number of vertices are initially infected, and then, for a fixed constant $\rho\in[0,1]$, after each time step an uninfected vertex will become infected if at least a $\rho$-proportion of its neighbours are already infected. Depending on the value of $\rho$ and the choice of initially infected vertices, this process may result in the infection spreading to the entire graph. We are primarily interested in identifying, for a fixed graph and value of $\rho$, the size of the smallest set that can be chosen to be initially infected which will result in the eventual infection of the entire graph.

This can be seen as a variant of the {\it bootstrap percolation process} which was first introduced by Chalupa, Leath, and Reich~\cite{ChLeRe79} as a simplification of the Ising model in statistical physics. For a fixed $r\in \mathbb N$, they considered the process in which vertices become infected when at least $r$ of their neighbours are infected, independently of their degree. Since then this case has often been called {\it $r$-neighbourhood bootstrap percolation}. Besides its applications in statistical physics, this threshold model was also used to describe problems in sociology, epidemiology, economics and computer science such as the diffusion of information and opinion in social networks, the spread of diseases, and cascading failures in financial or computer networks (see for example~\cite{BlEaKlKlTa11, DrRo09, KeKlTa03} and the references therein). 

The idea for a degree-proportional model dates back to Granovetter~\cite{Gr78}, based on the natural assumption that an individual may be more likely to be swayed by finding that a significant proportion of their friends hold a particular view, rather than simply encountering a fixed number of people who have adopted that position. Sociologists call this process {\it complex propagation} in contrast to {\it simple propagation} in which one active neighbour is enough to influence a vertex. For studies investigating this complex model, see for example \cite{CeMa07, CeEgMa07, Wa02}.

Formally, we define the following more general framework. Let $G$ be a graph and $\vphi:V(G)\rightarrow \mathbb R_{\geq 0}$ be a function. For a set $A\subseteq V(G)$ we define a set $H_{\vphi,G}(A)\subseteq V(G)$ iteratively in the following way. We set $A_0:=A$ and $A_{i+1}:=A_i\cup \{v\in V(G) : |N_G(v)\cap A_i|\geq \vphi(v)\}$. We call $(A_i)_{i\in\mathbb N}$ the {\it infecting sequence} in $G$ induced by $A$ with respect to $\vphi$. Note that $A_i\subseteq A_{i+1}$ for every $i\in\mathbb N$. Furthermore, if $A_i=A_{i+1}$ for some $i\in\mathbb N$, then $A_{i+1} = A_{i+2}$. So for a finite graph $G$ there must exist $k \in \mathbb N$ such that $A_k = A_{k+1} = A_{k+2} = \cdots$, and we set $H_{\vphi,G}(A):=A_k$ for this $k$. For simplicity we will write $H_{\vphi,G}(v)$ in place of $H_{\vphi,G}(\{v\})$ for a vertex $v\in V(G)$. We say that a vertex $v$ {\it becomes infected at step $i$}, if $v\in A_{i}\setminus A_{i-1}$, and for a set $S\subseteq H_{\vphi,G}(A)$ we say that $A$ {\it infects} $S$. We now call a set $A\subseteq V(G)$ a {\it contagious set} for $G$ with respect to $\vphi$ if $H_{\vphi,G}(A)=V(G)$. In the literature infected vertices are also called {\it active} and it is said that $A$ {\it percolates}, if $A$ is a contagious set. Sometimes a contagious set is also called a {\it perfect target set} or an {\it irreversible dynamic monopoly}; the latter in contrast to the different process of {\it dynamic monopolies} in which active vertices may become inactive again as investigated in~\cite{Pe98} and~\cite{Be01}. The special case of the constant function $\vphi(v)\equiv r$ is precisely $r$-neighbourhood bootstrap percolation.

A major focus in percolation theory has been to approximate, for various families of graphs, the {\it critical probability} which is the infimal density at which an independent random selection of initially infected vertices with this density is likely to be contagious. In an abundant line of research this problem was investigated for multidimensional grids~\cite{AiLe88, BaPe98, CeCi99, CeMa02, Ho03, BaBoMo09a, BaBoMo10, GrHoMo12, BaBoDuMo12}, the hypercube~\cite{BaBo06}, the Hamming torus~\cite{GrHoPfSi15} and various distributions of random graphs~\cite{BaPi07, JaLuTuVa12, AmFo14, KaMa16}. An example of a degree-proportional function $\vphi$ which was investigated is the so-called {\it majority bootstrap percolation}, i.e. $\vphi(v)=d_G(v)/2$; for this function the critical probability was studied for the hypercube in~\cite{BaBoMo09b}.

Another research direction is the extremal problem of bounding the minimal size of a contagious set in the bootstrap percolation process. For a given function $\vphi$ and a graph $G$ we denote this number by {\it $h_{\vphi}(G)$}. In the case of the $r$-neighbourhood bootstrap percolation this was again investigated for special families of graphs. Inspired by a Hungarian folklore problem Pete formulated this problem for finite multidimensional grids and gave first bounds on the size of a minimal contagious sets (see Balogh and Pete~\cite{BaPe98} for a summary of these results). The case $r=2$ was settled by Balogh and Bollob{\'a}s~\cite{BaBo06}. Recently exact asymptotics for the hypercube were determined by Morrison and Noel~\cite{MoNo16}. This proved a conjecture of Balogh and Bollob{\'a}s~\cite{BaBo06} and improved earlier bounds of Balogh, Bollob{\'a}s and Morris~\cite{BaBoMo10}. Additionally, Morrison and Noel also proved bounds for multidimensional rectangular grids. Furthermore, the minimal size of a contagious set has been investigated for expander graphs~\cite{CoFeKrRe15}, very dense graphs~\cite{FrPoRe15}, random graphs~\cite{JaLuTuVa12, AmFo12, FeKrRe16, AnKo17}, with additional Ore and Chv{\'a}tal-type degree conditions~\cite{DaFeLiMaPfUz16}, and in the setting of hypergraphs~\cite{BaBoMoRi12}.

A general bound on the minimal size of a contagious set for every function $\vphi$ and graph $G$ was provided by Ackermann, Ben-Zwi and Wolfovitz~\cite{AcBeWo10} who showed that
\begin{equation}
h_{\vphi}(G)\leq\sum_{v\in V(G)}\frac{\lceil\vphi(v)\rceil}{d_G(v)+1}\;.\label{eq:acker}
\end{equation}
This bound was also independently derived by Reichmann~\cite{Re12} for the special case of $r$-neighbourhood bootstrap percolation. Very recently this bound was improved by Cordasco, Gargano, Mecchia, Rescigno and Vaccaro in \cite{CoGaMeReVa16} to
\begin{equation}
h_{\vphi}(G)\leq\sum_{v\in V'}\min\left\{1,\frac{\lceil\vphi(v)\rceil}{d'(v)+1}\right\}\;,\label{eq:acker_improved}
\end{equation}
where $V'=\{v\in V(G)\mid d_G(v)\geq 2 \mbox{ or }\vphi(v)\neq 1\}$ and $d'(v) = |N_G(v)\cap V'|$.

In this paper we focus on general upper bounds on the size of a smallest contagious set in a degree-proportional bootstrap percolation process, i.e. $\vphi_{\rho}(v)=\rho d_G(v)$, where $\rho \in[0,1]$ is a fixed constant. Observe that if $G$ is not connected, then a smallest contagious set of $G$ is simply the union of smallest contagious sets, one from each connected component, so we need only consider connected graphs. Throughout this paper we will often shorten the notation of $\vphi_{\rho}$ to simply $\rho$, so for example we will write $h_{\rho}(G)$ instead of $h_{\vphi_{\rho}}(G)$. 

For majority bootstrap percolation (where $\rho = 1/2$), expression~\eqref{eq:acker} shows that for any graph $G$ we have $h_{1/2}(G) \leq |V(G)|/2$. This bound is best possible, as shown by taking $G$ to be a clique of even order. However, expression~\eqref{eq:acker} and also expression~\eqref{eq:acker_improved} may be far from best possible if $\rho$ is small and $G$ has many vertices of small degree. For example, for $1/n < \rho \ll 1$, let $K_{1,n-1}$ be the star on $n$ vertices with partite sets $\{v\}\cup B$. We add a perfect matching to $B$ and call the resulting graph $S$. Note that according to the notation of expression~\eqref{eq:acker_improved} we have $V'=V(S)$ and $d'(v)=d_S(v)$. Hence
$$h_{\rho}(S)=1<\rho n\ll \frac{n-1}{3}<\sum_{v\in V'}\frac{\lceil\rho d_S(v)\rceil}{d'(v)+1}\;.$$
For this reason attention has recently focused on providing different kinds of upper bounds which might work better for small values of $\rho$; in particular on upper bounds of the form $h_{\rho}(G) < C \rho n$ for some constant $C$. However, since every contagious set has size at least one, such bounds must necessarily exempt the case when $h_\rho(G) = 1$ (some previous works have expressed this exemption by a lower bound on the maximum degree of $G$ or the order of $G$, but these are easily seen to be equivalent). 
So, more precisely,  
the aim is to determine the smallest constant $C$ such that every connected graph $G$ fulfils  $h_{\rho}(G) < C \rho n$ or $h_{\rho}(G)=1$. Chang and Lyuu~\cite{ChLy15} observed that the complete graph provides a lower bound of order $O(\rho n)$ on $h_{\rho}(G)$. Indeed, for $n := \lfloor 1/\rho\rfloor+2$, every contagious set for $K_{n}$ has size at least $2\geq 2\rho n-4\rho$, from which we derive the following observation which shows that we cannot achieve $C<2$.

\begin{obs}\label{obs:example}
For every $\eps>0$ there exists $\rho>0$ and a connected graph $G$ on $n$ vertices such that $h_{\rho}(G)\neq 1$ and
$$h_{\rho}(G)\geq (2-\eps)\rho n\;.$$
\end{obs}

The first bound of the form described above was due to Chang~\cite{Ch11}, who showed that every connected graph $G$ on $n$ vertices satisfies $h_{\rho}(G)=1$ or
\begin{equation}
h_{\rho}(G)\leq (2\sqrt{2}+3)\rho n<5.83\rho n\;.\label{eq:chang}
\end{equation}
Chang and Lyuu~\cite{ChLy15} improved this bound by showing that every such $G$ satisfies $h_{\rho}(G)=1$ or
\begin{equation*}
h_{\rho}(G)\leq 4.92\rho n\;.\label{eq:changlyuu}
\end{equation*}
This was the best general bound prior to this work. Very recently Gentner and Rautenbach~\cite{GeRa17} provided an upper bound which asymptotically matches the lower bound of Observation~\ref{obs:example} under the additional assumptions that $G$ has girth at least five and $\rho$ is sufficiently small. More precisely, they showed that for any $\eps > 0$, if $\rho$ is sufficiently small and $G$ has girth at least five, then $h_{\rho}(G)=1$ or
\begin{equation} \label{eq:GR}
h_{\rho}(G)\leq(2+\eps)\rho n\;.
\end{equation}
The first main result of this paper is the following theorem, which improves on each of the aforementioned results by giving the optimal bound of the described form.

\begin{thm}\label{thm:genbound}
For every $\rho\in (0,1]$, every connected graph $G$ of order $n$ has $h_{\rho}(G)=1$ or
\[h_{\rho}(G)< 2\rho n\;.\]
\end{thm}

In particular, for connected graphs of order $n > 1/(2\rho)$ we always have $h_{\rho}(G)< 2\rho n$.
The second main result of this paper is a  stronger bound for graphs of girth at least five, obtained by combining our approach for Theorem~\ref{thm:genbound} with the method used by Gentner and Rautenbach~\cite{GeRa17} to prove~\eqref{eq:GR}.

\begin{thm}\label{thm:girthbound}
For every $\eps>0$ there exists $\rho_0>0$ such that for every $\rho\in(0,\rho_0)$ and every connected graph $G$ of order $n$ and girth at least $5$ we have $h_{\rho}(G)=1$ or
\[h_{\rho}(G) < (1+\eps)\rho n\;.\]
\end{thm}

Theorem~\ref{thm:girthbound} is also asymptotically best possible, as shown by the following example of the balanced $(\lfloor1/\rho\rfloor+1)$-regular tree. 

\begin{obs}\label{obs:exampleTree}
For every $\eps>0$, every $\rho_0>0$ and every $N_0\in\mathbb N$ there exists $\rho\in(0,\rho_0)$ and a connected graph $G$ with girth at least $5$ on $n>N_0$ vertices such that
$$h_{\rho}(G)> (1-\eps)\rho n\;.$$
\end{obs}

\begin{proof}
Given $\eps, \rho_0$ and $N_0$, choose $\rho \in (0, \rho_0)$ sufficiently small so that $d:=\lfloor1/\rho\rfloor+1$ fulfils $d>\max\{2/\eps,\sqrt{N_0-1}\}$. Let $G$ be the balanced $d$-regular tree with root $v$ of order $n:=1+d^2>N_0$, which consists of three `tiers' $T_1=\{v\},T_2=N_G(v)$ and $T_3= V(G) \sm (\{v\} \cup N_G(v))$ such that $|T_2|=d$ and $|T_3|=d(d-1)$. We will show that $h_{\rho}(G)=d$ and therefore that
$$h_{\rho}(G)=d=\frac{d}{\rho (1+d^2)}\rho n\geq \frac{(d-1)d}{1+d^2}\rho n=\left(1-\frac{d+1}{1+d^2}\right)\rho n>(1-2/d)\rho n>(1-\eps)\rho n\;.$$
First note that $T_2$ infects the whole graph $G$, so $h_{\rho}(G) \leq d$. On the other hand, observe that every vertex in $T_2$ needs at least two infected neighbours in order to become infected itself. For each $u \in T_2$ write $T_u:=\{u\}\cup N_G(u)\setminus\{v\}$. Then, because of our previous observation, every contagious set $A \subseteq V(G)$ fulfils $|A\cap T_u|\geq 1$ for each $u\in T_2$. Since the sets $T_u$ are pairwise disjoint it follows that $|A|\geq |T_2|=d$, so $h_{\rho}(G) \geq d$.
\end{proof}

Finally, we note that the extremal construction for Observation~\ref{obs:example} relies on the fact that $n$ is chosen close to $1/\rho$. So it is natural to ask, for fixed $\rho \in (0, 1]$, whether stronger bounds could be given for sufficiently large $n$. We give a detailed discussion of this aspect in Section~\ref{sec:remarks}. In particular, we construct arbitrary large graphs with $h_{\rho}(G)\geq(1+c_{\rho})\rho n$ where $c_\rho > 0$ is a fixed constant depending only on $\rho$, showing that we cannot hope for a bound of the form $(1+o_n(1)) \rho n$.

The next section is the main part of the paper containing the proofs of Theorem~\ref{thm:genbound} and Theorem~\ref{thm:girthbound}. We conclude with the discussion of some related open problems in the last section of the paper. Given a graph $G$ we write $n(G)$ for the order of $G$, and for any subset $S \subseteq V(G)$ we write $G-S$ for the induced subgraph $G[V(G) \sm S]$.

\section{Proofs of the theorems}
\subsection{Proof of Theorem~\ref{thm:genbound} }
The underlying idea we use to prove Theorem~\ref{thm:genbound} is the following. First, we suppose the existence of a minimal counterexample to our assertion and investigate its structure, categorising the vertices according to their degree in relation to $\rho$. In this way we can distinguish between vertices according to the number of infected neighbours necessary to cause their own infection. Next, using the minimality of this counterexample, we derive several degree conditions; in particular we bound the number of edges between low degree and high degree vertices. Finally, we show that we can infect the whole graph by a certain random selection of vertices just among the vertices of high degree. This will contradict the existence of a minimal counterexample. Our random selection follows the approach of Ackerman, Ben-Zwi and Wolfovitz~\cite{AcBeWo10} used to establish the bound~\eqref{eq:acker}, whilst the idea of distinguishing between low and high degree vertices was used by~Chang~\cite{Ch11} to prove the bound~\eqref{eq:chang}. The principal difference in our proof is that our analysis of the degree conditions is more detailed, as we distinguish more classes of degree, and that we take advantage of the structure of a minimal counterexample. During the proof we will work with the following easy observation.

\begin{prop}\label{prop:activate}
Suppose that $\rho\in[0,1]$ and $G$ is a graph. Let $S\subseteq V(G)$ and set $G'=G-S$. If $A\subseteq V(G')$ is a contagious set for $G'$ with respect to $\rho$, then $A\cup S$ is a contagious set for $G$ with respect to $\rho$.
\end{prop}

\begin{proof}
Let $(X_i)$ be the infecting sequence in $G$ induced by $A\cup S$ with respect to $\rho$ and let $(Y_i)$ be the infecting sequence in $G'$ induced by $A$ with respect to $\rho$. We show by induction on $i$ that $Y_i\cup S\subseteq X_i$ and therefore $V(G)=H_{\rho,G'}(A)\cup S\subseteq H_{\rho,G}(A\cup S)$. For the base case we certainly have $Y_0\cup S=A\cup S=X_0$. So assume the statement is true for $i\in \mathbb N$. Let $v\in Y_{i+1}\setminus Y_{i}$. We then have that $|N_{G'}(v)\cap Y_i|\geq \rho d_{G'}(v)$. We show that $|N_{G}(v)\cap X_i|\geq \rho d_G(v)$ and therefore $v\in X_{i+1}$. For this define $M(v):=N_G(v)\cap S$ and note that $N_G(v)=N_{G'}(v)\cupdot M(v)$ by the definition of $G'$. In particular we have that $d_G(v)=d_{G'}(v)+ |M(v)|$. Since $S\cup (N_{G'}(v)\cap Y_i)\subseteq X_i$ by the induction hypothesis, we can deduce that
$$|N_G(v)\cap X_i|\geq |M(v)|+|N_{G'}(v)\cap Y_i|\geq |M(v)|+\rho d_{G'}(v)\geq\rho(|M(v)|+d_{G'}(v))=\rho d_G(v)\;.$$

\end{proof}

We can now turn to the proof of Theorem~\ref{thm:genbound}. 

\begin{proof}[Proof of Theorem~\ref{thm:genbound}]
Observe first that if $G$ is a connected graph of order $n \leq 1/\rho + 1$ then $h_\rho(G) = 1$. Indeed, every vertex $v$ in such a graph has $d(v) \leq 1/\rho$ and so would require only one infected neighbour for its infection; since $G$ is connected it would then follow that any vertex of $G$ forms a contagious set of size one, so $h_{\rho}(G)=1$.

It therefore suffices to prove that every connected graph $G$ of order $n > 1/(2\rho)$ fulfils $h_{\rho}(G) < 2\rho n$, and the rest of the proof is devoted to proving this statement by considering a minimal counterexample (the apparently weaker-than-necessary bound on $n$ in this statement will prove convenient in appealing to minimality). We may assume for this that $\rho \in (0,1/2)$, as otherwise the statement is trivial. So fix $\rho \in (0, 1/2)$, and let $\mathcal{G}$ be the set of \emph{counterexamples} for this value of $\rho$ (in other words, $\mathcal{G}$ is the set of all connected graphs $G$ with $n(G) > 1/(2\rho)$ such that $h_{\rho}(G)\geq 2\rho n(G)$). Suppose for a contradiction that $\mathcal{G}$ is non-empty (i.e. that a counterexample exists). Let $n = \min_{G \in \mathcal{G}} n(G)$, and choose $G \in \mathcal{G}$ with $n(G) = n$ such that $e(G) \leq e(H)$ for every $H\in \mathcal{G}$ with $n(H) = n$. Thus $G$ is a minimal counterexample according to its number of vertices and (among those) to its number of edges. Moreover, we must have
\begin{equation}
n > 1/\rho+1\;,\label{eq:minvert}
\end{equation}
as otherwise our initial observation and the fact that $n > 1/(2\rho)$ together imply that $h_\rho(G) = 1 < 2 \rho n$, so $G$ would not be a counterexample.
The vertex-minimality of $G$ will be used to prove the following claim.

\begin{claim}\label{clm:vertex}
For every $v\in V(G)$ we have that $|H_{\rho,G}(v)|< 1/(2\rho)$.
\end{claim}

\begin{proof}
For the sake of a contradiction assume that there exists $v\in V(G)$ such that $|H_{\rho, G}(v)|\geq 1/(2\rho)$. We set $G':=G-H_{\rho, G}(v)$ and denote by $\mathcal{C}$ the set of connected components of~$G'$. Note first that $\{v\}$ cannot infect all vertices of $G$, as otherwise we would have $h_{\rho}(G)=1=2\rho\cdot 1/(2\rho)<2\rho n$ by~\eqref{eq:minvert}, contradicting our assumption that $G$ is a counterexample. So $\mathcal{C}$ is non-empty. Now assume that there is a component $C\in\mathcal{C}$ which contains at most $1/(2\rho)$ vertices. As $G$ is connected, there exists a vertex $u \in V(C)$ with at least one neighbour in $H_{\rho, G}(v)$. Define $m(u):=|N_G(u)\cap H_{\rho,G}(v)|$, so $m(u) \geq 1$. Since $u\notin H_{\rho, G}(v)$, we derive the inequality
\[1/(2\rho)+m(u) > |N_G(u)\cap C|+m(u)=d_G(u)>1/\rho \cdot m(u)\;.\]
As $m(u) \geq 1$, this leads to $\rho >1/2$, which is a contradiction. Therefore each connected component of $G'$ contains more than $1/(2\rho)$ vertices. Since $G$ was a minimal counterexample, it follows that for each $C\in\mathcal{C}$ there exists a subset $A_C\subseteq V(C)$ with $|A_C|<2\rho |V(C)|$ which is a contagious set, that is, such that $H_{\rho,C}(A_C)=V(C)$. We then have $H_{\rho,G'}(\bigcup_{C\in\mathcal{C}}A_C)=V(G')$. Since $v$ infects $H_{\rho, G}(v)$ and, by Proposition~\ref{prop:activate} we have that $\bigcup_{C\in\mathcal{C}}A_C\cup H_{\rho, G}(v)$ infects $G$, it follows that $H_{\rho,G}(\bigcup_{C\in\mathcal{C}}A_C\cup\{v\})=V(G)$. Hence
$$h_{\rho}(G)\leq1+\sum_{C\in\mathcal C}|A_C|<2\rho\left(|H_{\rho, G}(v)|+\sum_{C\in\mathcal C} |V(C)|\right)\leq2\rho n\;,$$
contradicting our assumption that $G$ is a counterexample.
\end{proof}

We now classify the vertices of $G$ according to their degree by defining 
\begin{align*}
&V_1:=\{v\in V(G)\mid d_G(v)\leq 1/\rho\}, \\
&V_{\geq 2}:=V(G)\setminus V_1\textrm{ , and }\\
&V_2:=\{v\in V_{\geq 2}\mid d_G(v)\leq 2/\rho\}\;.
\end{align*}

Note that $V_1$ is the set of vertices which become infected if just one of their neighbours is infected, $V_{\geq 2}$ is the set of vertices which cannot be infected by just one infected neighbour, and $V_2$ is the set of vertices which can be infected by two infected neighbours, but not by just one infected neighbour. Claim~\ref{clm:vertex} together with the minimality of $G$ will imply the following claim about the size of neighbourhoods of vertices in these sets.
\begin{claim}\label{clm:vertdeg} The following holds.
\begin{enumerate}[noitemsep, label=(\roman*)]
 \item $|N_G(v)\cap V_1|< 1/(2\rho)$ for every $v\in V$. \label{clm:degv1}
 \item The induced subgraph $G[N(v)\cap V_2]$ is a clique for every $v \in V_{1}$. \label{clm:clique}
 \item $|N_G(v)\cap V_2|<1/(2\rho)$ for every $v\in V_1$. \label{clm:degv2}
\end{enumerate} 
\end{claim}

\begin{proof}
To prove~\ref{clm:degv1} assume for a contradiction that there exists $v\in V$ such that $|N(v)\cap V_1|\geq 1/(2\rho)$. If $v$ is initially infected, then in the first time step every vertex of $N(v) \cap V_1$ will become infected. So $|H_{\rho, G}(v)|\geq |\{v\}\cup (N(v)\cap V_1)|> 1/(2\rho)$ contradicting Claim~\ref{clm:vertex}. 

For~\ref{clm:clique}, assume for a contradiction that there are distinct $x,y \in N(v)\cap V_2$ such that $\{x,y\}\notin E(G)$.  
Let $H$ be the graph formed from $G$ by removing the edges $\{v,x\}$ and $\{v,y\}$ and adding the edge $\{x,y\}$, and let $C$ be the connected component of $H$ containing $x$ and~$y$. 

We claim that every contagious set $A \subseteq V(C)$ for $C$ also infects $V(C)$ in $G$.
To see this, let $(X_i)$ be the infecting sequence in $G$ induced by $A$ with respect to $\rho$ and let $(Y_i)$ be the infecting sequence in $C$ induced by $A$ with respect to $\rho$. We now prove by induction on $i$ that $Y_i \subseteq X_{2i}$ for each $i\in\mathbb N$. The base case is given by $Y_0=A=X_0$, so assume that $Y_i\subseteq X_{2i}$ for some $i\in\mathbb N$ and consider the following cases for each $w \in Y_{i+1}\setminus Y_{i}$. 
\begin{enumerate}[label=(\alph*)]
\item If $w\notin\{x,y,v\}$, then by construction of $H$ we have $N_G(w)=N_{C}(w)$, and by our induction hypothesis we can conclude that at least $\rho d_G(w)=\rho d_{C}(w)$ neighbours of $w$ are in $Y_i\subseteq X_{2i}$. Hence $w\in X_{2i+1}\subseteq X_{2(i+1)}$ as required. 
\item If $w=v$, then $v\in V(C)$ and by construction of $H$ we have that $N_{C}(v)\subseteq N_G(v)$. Again since $v\in Y_{i+1}$, we have that $\rho d_{C}(v)\leq |N_{C}(v)\cap Y_i|$ and therefore by also making use of the induction hypothesis there exists $u\in (N_{C}(v)\cap Y_{i})\subseteq (N_G(v)\cap X_{2i})$. Since $d_G(v)\leq 1/\rho$, we have that $|N_G(v)\cap X_{2i}|\geq 1\geq \rho d_G(v)$ and therefore $v\in X_{2i+1}\subseteq X_{2(i+1)}$. 
\item If $w\in\{x,y\}$, then we may assume without loss of generality that $w=x$ (since the following arguments hold equally well with $y$ in place of $x$). Note that by construction of $H$ we have that $N_{C}(x)=(N_G(x)\setminus\{v\})\cup\{y\}$ and in particular that $d_G(x)=d_{C}(x)$. Furthermore, since $x\in Y_{i+1}\setminus Y_i$, we have that $|N_{C}(x)\cap Y_i|\geq \rho d_{C}(x)$. Setting $U:=N_{C}(x)\setminus\{y\}$, we first assume that $|U\cap Y_i|\geq \rho d_G(x)$. By the induction hypothesis $U\cap Y_i\subseteq X_{2i}$ and since $U\subseteq N_G(x)$ we have that $x\in X_{2i+1}\subseteq X_{2(i+1)}$. On the other hand, if $|U\cap Y_i|<\rho d_G(x)$, then $y\in Y_i\subseteq X_{2i}$. Since $d_G(v)\leq 1/\rho$ and $y\in N_G(v)\cap X_{2i}$, we have that $v\in X_{2i+1}$. Furthermore $v\notin U$ and $U\subseteq X_{2i}\subseteq X_{2i+1}$. Hence $|N_G(x)\cap X_{2i+1}|\geq |(U\cap Y_i)\cup \{v\}|\geq \rho d_{C}(x)=\rho d_G(x)$ and therefore $x\in X_{2(i+1)}$. 
\end{enumerate}
This completes the induction argument, so we have $Y_i \subseteq X_{2i}$ for each $i\in\mathbb N$, and therefore $V(C)=H_{\rho,C}(A)\subseteq H_{\rho,G}(A)$. This proves our claim that every contagious set $A \subseteq V(C)$ for $C$ also infects $V(C)$ in $G$.

Since $G$ is connected, the construction of $H$ ensures that $H$ has at most two connected components. Suppose first that $H$ is connected and therefore that $H=C$. Since we assumed that $G$ has no contagious set of size less than $2 \rho n$, it follows from our above claim that $C$ has no contagious set of size less than $2 \rho n$. 
Thus $C = H$ is also a counterexample, contradicting the edge-minimality of $G$. 
We conclude that $H$ must have exactly two connected components $C$ and $C'$.  
Observe that since $d_H(x)=d_G(x)> 1/\rho$ we have that $|V(C)|>1/(2\rho)$.
Suppose now that $|V(C')|\leq 1/(2\rho)$. Note that then every vertex $w \in V(C')$ has $d_G(w) \leq 1/\rho$. Indeed, if $w=v$ then this is implied by $v\in V_1$ and if $w\neq v$ then $w$ cannot have any neighbour in $C$, as otherwise $H$ would be connected, and therefore $|N_G(w)|\leq |C'|\leq 1/\rho$. Now let $A \subseteq V(C)$ be a contagious set for $C$. Then by our above claim $A$ infects $V(C)$ in $G$, so $V(G)\setminus H_{\rho,G}(A) \subseteq V(C')$. It follows that every vertex $w \in V(G)$ not infected by $A$ has $d_G(w) \leq 1/\rho$, and since $G$ is connected it follows that $H_{\rho,G}(A)=V(G)$. In other words any contagious set for $C$ is also a contagious set for $G$, and therefore $h_{\rho}(C)\geq 2\rho |V(C)|$. Thus $C$ is also a counterexample, contradicting the vertex-minimality of $G$. 

We therefore conclude that $C$ and $C'$ each have more than $1/(2\rho)$ vertices. By minimality of $G$ it follows there are contagious sets $A \subseteq V(C)$ and $A' \subseteq V(C')$ for $C$ and $C'$ respectively with $|A| < 2\rho |V(C)|$ and $|A'| < 2 \rho |V(C')|$. Then $A$ infects $V(C)$ in $G$ by our above claim. Moreover, $A'$ infects $V(C')$ in $G$ since every vertex $w \in V(C')$ with $w \neq v$ has $N_G(w) = N_{C'}(w)$, and $d_G(v) < 1/\rho$. We conclude that $A \cup A'$ is a contagious set in $G$ of size $|A \cup A'| \leq |A| + |A'| < 2\rho |V(C)| + 2 \rho |V(C')| = 2\rho n$, contradicting our assumption that $G$ was a counterexample and so completing the proof of~\ref{clm:clique}.

Finally, to prove~\ref{clm:degv2}, assume for a contradiction that there exists $v\in V_1$ such that $|N(v)\cap V_2|\geq 1/(2\rho)$, and choose any $w \in N(v)\cap V_2$. Since $G[N(v)\cap V_2]$ is a clique by~\ref{clm:clique}, every other vertex of $N(v) \cap V_2$ is a neighbour of both $v$ and $w$. Since each vertex of $V_2$ will become infected once it has two infected neighbours, it follows that if $w$ is initially infected, then $v$ will become infected at the first time step, and then all remaining vertices of $N(v) \cap V_2$ will become infected at the second time step. So $|H_{\rho, G}(w)|\geq |\{v\}\cup (N(v)\cap V_2)|\geq 1/(2\rho)$, contradicting Claim~\ref{clm:vertex}.
\end{proof}

We now introduce some further notation. 
Let $M$ be the number of edges with one endvertex in $V_1$ and one endvertex in $V_2$.
Also, for each $v \in V(G)$, write $d'(v) := |N(v)\cap V_1|$ and define $x_v := \rho d(v)$ and $y_v := \rho d'(v)$.
Note that, by Claim~\ref{clm:vertdeg}\ref{clm:degv1}, for every $v\in V_{\geq 2}$ we have
\begin{equation}
0\leq y_v< 1/2\;.\label{eq:yv}
\end{equation}
Furthermore, double-counting the edges between $V_1$ and $V_2$ together with Claim~\ref{clm:vertdeg}\ref{clm:degv2} gives us
\begin{equation}
\sum_{v\in V_2}y_v = \rho \sum_{v\in V_2} d'(v)= \rho M= \rho \sum_{v\in V_1}|N(v)\cap V_2| < \tfrac{1}{2}|V_1|\;.\label{eq:sumdegv1}
\end{equation}
We now choose an order $\sigma$ of the vertices of $V_{\geq 2}$ uniformly at random. Let $A$ be the set of vertices $v\in V_{\geq 2}$ for which fewer than $\lceil\rho d(v)\rceil$ neighbours of $v$ precede $v$ in the order $\sigma$. We claim that $A$ is a contagious set for $G$ with respect to $\rho$. Indeed, since $G$ is connected and each vertex of $V_1$ becomes infected once it has one infected neighbour, it suffices for this to show that $A$ infects $V_{\geq 2}$. Suppose for a contradiction that this is not the case, and let $v \in V_{\geq 2}$ be the first vertex in the order $\sigma$ which is not infected by $A$. Then $v \notin A$, so $v$ has at least  $\lceil\rho d(v)\rceil$ neighbours which precede $v$. Our choice of $v$ implies that all of these neighbours are infected by $A$, but this contradicts the fact that $v$ is not infected by $A$. So $A$ is indeed a contagious set for $G$ with respect to $\rho$. Moreover, for any $v \in V_{\geq 2}$ we have $v \in A$ if and only if $v$ is one of the first $\lceil\rho d(v)\rceil$ members of $\{v\} \cup (N(v) \cap V_{\geq 2})$ in the order $\sigma$. It follows that $\mathbb P[v\in A]=\frac{\lceil\rho d(v)\rceil}{1+|N(v)\cap V_{\geq 2}|}$, so the expected size of $A$ is
\begin{align*}
\Expect[|A|]&=\sum_{v\in V_{\geq 2}}\frac{\lceil\rho d(v)\rceil}{1+d(v)-d'(v)} < \sum_{v\in V_{\geq 2}}\frac{1+\rho d(v)}{d(v)-d'(v)}
= \rho\sum_{v\in V_{\geq 2}}\frac{1+x_v}{x_v-y_v}\\
&= \rho\left(\sum_{v\in V_2}\frac{1+x_v}{x_v-y_v}+\sum_{v\in V_{\geq 2}\setminus V_2}\frac{1+x_v}{x_v-y_v}\right) \\
&\overset{(*)}{\leq}\rho\left(\sum_{v\in V_2}\frac{1+1}{1-y_v}+\sum_{v\in V_{\geq 2}\setminus V_2}\frac{1+2}{2-1/2}\right)\\
&= 2\rho\left(\sum_{v\in V_2}\frac{1}{1-y_v}+|V_{\geq 2}\setminus V_2|\right)=:T\;,
\end{align*}
where the inequality $(*)$ follows from~\eqref{eq:yv}, the fact that $x_v \geq 1$ for $v \in V_2$ and $x_v \geq 2$ for $x \in V_{\geq 2} \sm V_2$, and the fact that if $0\leq y<a\leq x$ are real numbers, then $\frac{1+x}{x-y}\leq\frac{1+a}{a-y}$. For the final argument we will use the following claim.

\begin{claim}\label{clm:max}
Let $n_1,n_2\in \mathbb N$, $K=\{y\in [0,\frac{1}{2}]^{n_2}\mid \sum_{i=1}^{n_2} y_i\leq\frac{n_1}{2}\}$ and $f(y)=\sum_{i=1}^{n_2} \frac{1}{1-y_i}$. We have that
$$\max_{y\in K} f(y)\leq n_1+n_2\;.$$
\end{claim}

\begin{proof}
Suppose first that $n_1 \geq n_2$, and let $u := (\frac{1}{2}, \dots, \frac{1}{2})$. Then $u \in K$, and it is easy to see that $f(u)\geq f(y)$ for every $y\in K$. Furthermore $f(u)=n_2\cdot\frac{1}{1-1/2}=2n_2\leq n_1+n_2$.

Now suppose instead that $n_1 < n_2$. Since $f$ is a continuous function it attains its maximum on the compact space $K$. Fix $u \in K$ such that this maximum is attained, and (subject to this) so that as many as possible of the $u_i$ are equal to $1/2$. Clearly we then have $\sum_{i=1}^{n_2} u_i = \frac{n_1}{2}$. Suppose first that there exist distinct $i, j \in [n_2]$ such that $0 < u_i \leq u_j < 1/2$. Then set $\eps:= \min\{u_i,1/2-u_j\}$ and define $z\in \mathbb R^{n_2}$ by 
$$z_k:=\begin{cases}u_i-\eps,\textrm{ if }k=i\\u_j+\eps,\textrm{ if }k=j\\u_k,\textrm{ otherwise}\end{cases}\;.$$
Note that $z\in K$. We then get by the convexity of the function $g(x) = \frac{1}{1-x}$ on $[0,\frac{1}{2}]$ that
$$f(u)= g(u_i)+g(u_j)+ \sum_{k\neq i,j} g(u_k) < g(u_i-\eps)+g(u_j+\eps) +\sum_{k\neq i,j}g(u_k)=f(z)\;.$$
This contradicts the assumption that $f(u)$ was the maximal value. We therefore conclude that all but at most one of the $u_i$ are equal to $0$ or $1/2$; since $n_1$ is an integer and $\sum_{i=1}^{n_2} u_i = \frac{n_1}{2}$, this implies that in fact every $u_i$ is equal to $0$ or $1/2$ (in other words $u$ has $n_1$ coordinates equal to $1/2$ and the rest equal to $0$). 
Hence $\max_{y\in K} f(y) = f(u) = n_1\cdot \frac{1}{1-1/2}+(n_2-n_1)\cdot \frac{1}{1-0}=n_1+n_2$.
\end{proof}

By setting $n_1=|V_1|$ and $n_2=|V_2|$ and by using~\eqref{eq:yv} and~\eqref{eq:sumdegv1} the assumptions of Claim~\ref{clm:max} are fulfilled and we can conclude our calculation with
\begin{align*}
\Expect[|A|]< T&\leq2\rho(|V_1|+|V_2|+|V_{\geq 2}\setminus V_2|)\\
&= 2\rho n\;.
\end{align*}
We may therefore fix an order $\sigma$ of $V_{\geq 2}$ such that the contagious set $A$ given by this order has size less than $2 \rho n$. This contradicts our assumption that $G$ was a counterexample and so proves the theorem.
\end{proof}

\subsection{Proof of Theorem~\ref{thm:girthbound}.}

To prove our second main theorem we use the following modified version of a theorem of Gentner and Rautenbach~\cite{GeRa17}. For a constant $\rho>0$ and a graph $G$ we again use the notation $V_1=\{v\in V(G)\mid d(v)\leq 1/\rho\}$ and $V_{\geq 2}=V\setminus V_1$.

\begin{lemma}[\cite{GeRa17}]\label{lem:girth}
For every $\eps>0$ there exists $\rho_0>0$ such that for any $\rho \in (0, \rho_0)$ the following holds. Let $G$ be a connected graph of order $n$ with girth at least $5$ and $\Delta(G)\geq 1/\rho$. If $|N(u)\cap V_1|<\frac{1}{1+\eps}d(u)$ for every $u\in V_{\geq 2}$, then 
\[h_{\rho}(G) < (1+\eps)\rho n\;.\]
\end{lemma}

The form of the lemma given above follows immediately from the proof of the original version~\cite[Theorem~1]{GeRa17}, which omitted the condition on the neighbourhood in $V_1$ of vertices in $V_{\geq 2}$, and had the weaker conclusion that $h_{\rho}(G)\leq (2+\eps)\rho n$.
Loosely speaking, in this proof Gentner and Rautenbach used a random selection process to select an initial set $X$ and showed that vertices satisfying certain properties are infected with high probability by $X$. They then bounded the number of vertices which do not have these properties, and simply added all such vertices to $X$ to obtain a contagious set of the claimed size.

More precisely, they first fixed a sufficiently small constant $\delta$, in particular with $\delta<\eps$, and found a set $X_0\subseteq V_{\geq 2}$ with two properties:
\begin{enumerate}[label=(\alph*)]
\item  $|X_0|\leq(1+\delta)\rho n$, and
\item $|N(u)\cap (V_1\setminus H_{\rho,G}(X_0))|\leq\frac{1}{1+\delta}d(u)$ for every $u\in V_{\geq 2}\setminus X_0$.
\end{enumerate}
Then, by a series of random selections which exploit the girth property to ensure independence of the random variables considered, they added at most $(1+\eps-\delta)\rho n$ vertices to $X_0$ to gain a set $X_0\subseteq Y\subseteq V(G)$ such that $|Y|\leq(2+\eps)n$ and $H_{\rho,G}(Y)=V(G)$. 

To prove Lemma~\ref{lem:girth} we follow exactly the same argument with $X_0 := \emptyset$; the additional condition of Lemma~\ref{lem:girth} ensures that $X_0$ satisfies~(a) and~(b). We then obtain a set $Y\subseteq V(G)$ such that $|Y| \leq (1+\eps-\delta)\rho n < (1+\eps) \rho n$ and $H_{\rho,G}(Y)=V(G)$. 

We can now prove Theorem~\ref{thm:girthbound}. The main idea is to show that a minimal counterexample must satisfy the conditions of Lemma~\ref{lem:girth}, and therefore is not in fact a counterexample.

\begin{proof}[Proof of Theorem~\ref{thm:girthbound}] 
Given $\eps >0$, let $\rho_0$ be small enough for Lemma~\ref{lem:girth} and also such that $\rho_0 <\frac{\eps}{1+\eps}$, and fix $\rho \in (0, \rho_0)$. 
We will show that if $G$ is a connected graph of order $n > \frac{1}{(1+\eps)\rho}$ with girth at least $5$, then $h_{\rho}(G)<(1+\eps) \rho n$. This suffices to prove Theorem~\ref{thm:girthbound} as, by a similar argument as in the proof of Theorem~\ref{thm:genbound}, every connected graph $G$ of order at most $\frac{1}{(1+\eps)\rho}$ has $h_{\rho}(G)=1$. 
So suppose that this assertion is false, let $G$ be a minimal counterexample according to the number of vertices (for these values of $\rho$ and $\eps$), and let $n$ be the order of $G$. Then $G$ is a connected graph on $n > \frac{1}{(1+\eps)\rho}$ vertices with girth at least five and $h_\rho(G) \geq (1+\eps) \rho n$. We may assume that $\Delta(G)> 1/\rho$, as otherwise (since $G$ is connected) we have $h_{\rho}(G)=1<(1+\eps)\rho n$, a contradiction. We show the following claim.
\begin{claim} \label{clm:girthdeg}
$|N(v)\cap V_1|< \frac{1}{(1+\eps)\rho}$ for every $v\in V_{\geq 2}$. 
\end{claim}

\begin{proof}
Suppose for a contradiction that there exists $v\in V_{\geq 2}$ such that $|N(v)\cap V_1|\geq \frac{1}{(1+\eps)\rho}$. Then $|H_{\rho, G}(v)|\geq |\{v\}\cup (N(v)\cap V_1)|> \frac{1}{(1+\eps)\rho}$. We set $G':=G-H_{\rho, G}(v)$ and denote by $\mathcal{C}$ the set of connected components of $G'$. 
Note first that $\{v\}$ cannot infect $V(G)$, as otherwise $h_{\rho}(G)=1<(1+\eps)\rho n$, a contradiction. Now assume that there is a component $C\in\mathcal{C}$ which contains at most $\frac{1}{(1+\eps)\rho}$ vertices. As $G$ is connected there exists a vertex $u\in V(C)$ with a neighbour in $H_{\rho,G}(v)$. Let $m(u):=|N(u)\cap H_{\rho,G}(v)|\geq 1$; since $u\notin H_{\rho, G}(v)$, we derive the inequality
\[\frac{1}{(1+\eps)\rho}+m(u) > d(u)>  m(u) \cdot \frac{1}{\rho}\;.\]
This leads to 
$$1 - \frac{\eps}{1+\eps} = \frac{1}{1+\eps}> (1-\rho) m(u)\geq 1-\rho \geq 1-\rho_0\;,$$
which contradicts our choice of $\rho_0$. Therefore each connected component of $G'$ contains more than $\frac{1}{(1+\eps)\rho}$ vertices. Since $G$ was a minimal counterexample, it follows that each component $C \in\mathcal{C}$ has $h_\rho(C) < (1+\eps)\rho |V(C)|$. By Proposition~\ref{prop:activate} we then have
$$h_{\rho}(G)\leq1+\sum_{C\in\mathcal C}h_{\rho}(C)<(1+\eps)\rho\left(|H_{\rho, G}(v)|+\sum_{C\in\mathcal C} |V(C)|\right) = (1+\eps)\rho n\;,$$
a contradiction. 
\end{proof}

So for every vertex $u\in V_{\geq 2}$ we have $|N(u)\cap V_1|<\frac{1}{(1+\eps)\rho}<\frac{1}{1+\eps}d(u)$. By Lemma~\ref{lem:girth}, it follows that $h_\rho(G) < (1+\eps) \rho n$. This contradicts our assumption that $G$ was a counterexample, and so completes the proof.
\end{proof}

\section{Further Remarks}\label{sec:remarks}

The example of the balanced $(\lfloor1/\rho\rfloor+1)$-regular tree in Observation~\ref{obs:exampleTree} shows that Theorem~\ref{thm:girthbound} is asymptotically best possible; there remains the question of whether it is possible to drop the $\eps$ in the statement of Theorem~\ref{thm:girthbound}. Note that the assumption about $\rho$ being sufficiently small is necessary for Theorem~\ref{thm:girthbound}. To see this, choose any $\rho\in (1/2,3/5)$, and observe that the $5$-cycle $C_5$ then fulfils $h_{\rho}(C_5)=3 >\rho n(C_5)$. By contrast Gentner and Rautenbach~\cite{GeRa17} showed that for every $\rho\in (0,1]$ and every tree $T$ of order at least $1/\rho$ we have
$h_{\rho}(T)\leq\rho n(T)\;.$
It would also be interesting to know whether Theorem~\ref{thm:girthbound} remains valid under the weaker assumption that $G$ has girth at least four (note that Observation~\ref{obs:example} shows that the girth assumption cannot be dropped altogether).

Recall that the example of the complete graph $K_{\lfloor 1/\rho \rfloor+2}$ shows that the constant 2 in Theorem~\ref{thm:genbound} cannot be improved. However, since that example has order roughly $1/\rho$, there remains the question of whether, for fixed $\rho$, the optimal bound for graphs on at least~$n$ vertices is asymptotically $\rho n$ as $n$ tends to infinity. The following construction shows that this is not the case for uncountably-many values of $\rho$, including values arbitrarily close to zero.

\begin{prop}\label{prop:example}
For every $\rho\in (0,1/3)$ such that $1/\rho-\lfloor1/\rho\rfloor>1/2$ there exists $\eps(\rho)>0$ such that for every $N_0\in\mathbb N$ there exists a connected graph $G$ of order $n\geq N_0$ with
$$h_{\rho}(G)> (1+\eps)\rho n\;.$$
\end{prop}

\begin{proof}
Roughly speaking, we use the integrality gap to produce a number of graphs $H_i$ each requiring two vertices to become infected, while containing slightly fewer than $2/\rho$ vertices. We connect these graphs together to a graph $G$ in a way that ensures that even if the whole graph $G$ except one $H_i$ is infected, we then still need at least two additional infected vertices from $H_i$ to infect the vertices in $H_i$. This ensures that a contagious set must contain at least two vertices from every $H_i$, while the size of these subgraphs bounds $h_{\rho}(G)$ away from $\rho n$.

Formally, set $\gamma=1/\rho-\lfloor1/\rho\rfloor-1/2$ and choose $k\in\mathbb N$ such that $k>\max\{1/(2\gamma),\rho N_0/2\}$ and $\eps\in\mathbb R$ such that $0<\eps<\frac{\rho(\gamma-1/(2k))}{1-\rho(\gamma-1/(2k))}$. Let $A$ and $B$ be two disjoint sets of $\lfloor1/\rho\rfloor$ vertices and $u$ be an additional vertex. We define a graph H with vertex set $A\cup B\cup\{u\}$ and edge set $\{\{v,w\}\mid v,w\in B \mbox{ or } v\in A,w\in B \mbox{ or } v=u,w\in A\}$. In other words, $B$ is a clique, and we add a complete bipartite graph from $A$ to $B$ and all edges from the vertices of $A$ to $u$. 

\begin{center}
\begin{figure}[h]
\begin{center}
\pgfdeclarelayer{background}
\pgfsetlayers{background,main}
\begin{tikzpicture}[scale=\genericScale]

		\def\z{5}
		\foreach \x in {1,2,...,\z}{
			\draw (\x*360/\z: 1cm) node (B\x) {};
			\fill (B\x.west) circle (\genericCirclePt);
		}

			\foreach \x in {1,2,...,\z}{
			\path (2,2*\x/\z -1.2) node (A\x) {};
		}

		\foreach \x in {1,2,...,\z}{
			\foreach \i in {1,2,...,\z}{
				\draw[line width=\genericThickness] (B\x.west)--(B\i.west);
			}
		}
		\begin{scope}

			\path [clip] (B4.south)--(B5.east)--(B1.north)--(A5.north)--(A1.south)--cycle;

			\foreach \x in {1,2,...,\z}{
				\foreach \i in {1,2,...,\z}{
				\draw[line width=\genericThickness][gray] (A\x.west)--(B\i.west);
			}
		}
		\end{scope}
	
		\path (3,0) node (u) {};
		\fill (u.west) circle (\genericCirclePt);
		
		\foreach \x in {1,2,...,\z}{
			\fill (A\x.west) circle (\genericCirclePt);
			\draw[line width=\genericThickness] (A\x.west)--(u.west);
		}

	\end{tikzpicture}
\caption{$H$ when $\lfloor1/\rho\rfloor=5$}
\label{fig:IntegralGapSubgraph}
\end{center}
\end{figure}
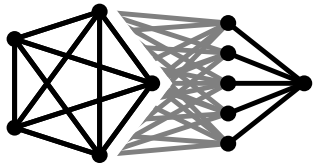
\end{center}

Let $H_1,\cdots,H_k$ be vertex disjoint copies of $H$ and denote by $u_1,\cdots,u_k$ the vertices corresponding to $u$. Let $s$ be another additional vertex. We define a graph $G$ to be the union of the $H_i$ and $\{s\}$ with the additional edges $\{\{s,u_i\}\mid i\in[k]\}$. Note that $n=|V(G)|=k|V(H)|+1\geq N_0$ and that a contagious set for $G$ contains at least $2$ vertices from each $H_i$. This follows since each vertex in $H_i$ requires two infected neighbours to become infected. The vertex $u$ may gain one infected neighbour if $s$ is infected but to become infected, at least one vertex from $A$ must also be infected. For a vertex from $A$ to become infected, then either two vertices from $B$ must be infected already, or one vertex from $B$ and $u$ must have been previously infected. Equally, no vertex from $B$ can become infected without two infected neighbours from within $A\cup B$. From this we see that only $u$ can become infected without two vertices from $H_i$ being in the contagious set, however the infection of $u$ requires an infected vertex in $A$ and will only lead to the infection of $H_i$ if there was also a vertex infected in $B$.
Thus we have
\begin{align*}
h_{\rho}(G)&\geq 2k=\frac{2k}{\rho(k|V(H)|+1)}\rho n
=\frac{2}{\rho(2\lfloor\frac{1}{\rho}\rfloor+1)+\rho/k}\rho n
= \frac{1}{\rho\lfloor\frac{1}{\rho}\rfloor+\rho/2+\rho/2k}\rho n\\
&= \frac{1}{1-\rho(\gamma-1/(2k))}\rho n
= \left(1+\frac{\rho(\gamma-1/(2k))}{1-\rho(\gamma-1/(2k))}\right)\rho n
> (1+\eps)\rho n\;.
\end{align*}
\end{proof}
Although this demonstrates that for a given $\rho$, it is not possible to give a bound tending to $\rho n$ as $n\rightarrow \infty$, it remains possible that a bound of the form $(\rho+o(\rho))n$ exists. In other words, it may be true that for sufficiently small~$\rho$, sufficiently large graphs have contagious sets whose size is close to $\rho n$. We formulate this formally as the following question.

\begin{qu}
Is it true that for every $\eps>0$ there exist $\rho_0\in (0,1]$ and $N_0\in\mathbb N$ such that for every $\rho\in(0,\rho_0]$ and every $n\geq N_0$ the following holds? If $G$ is a connected graph of order $n$, then
\[h_{\rho}(G)<(1+\eps)\rho n\;.\]
\end{qu}

\section*{Acknowledgements}

We would like to thank the anonymous referees for their careful reviews of the manuscript and for their many helpful comments.

\bibliographystyle{plain}
\bibliography{MinContagiousSet}

\end{document}